\documentclass{amsart}
\usepackage{amsmath, amsthm, amsfonts, amssymb}
\usepackage{mathrsfs,graphicx}
\providecommand{\noopsort[1]{}}
\usepackage{bbm}
\usepackage{dsfont}
\usepackage{ifthen}
\numberwithin{equation}{section}
\usepackage{a4}
\usepackage{hyperref}
\usepackage{enumerate}

\newtheorem{thm}{Theorem}[section]

\newtheorem{prop}[thm]{Proposition}
\newtheorem{lem}[thm]{Lemma}

\theoremstyle{remark}
\newtheorem{rem}[thm]{Remark}

\newtheorem{example}[thm]{Example}

\theoremstyle{definition}
\newtheorem{defn}[thm]{Definition}

\newcommand{\coloneqq}{\mathrel{\mathop:}=}

\newcommand{\applied}[2]{\langle #1,#2\rangle}

\newcommand{\dx}{\:\mathrm{d}}

\newcommand{\eps}{\varepsilon}
\DeclareMathOperator{\lh}{span}
\newcommand{\one}{\mathbbm{1}}

\newcommand{\R}{\mathds{R}}

\newcommand{\N}{\mathds{N}}
\newcommand{\Z}{\mathds{Z}}
\newcommand{\Q}{\mathds{Q}}

\newcommand{\fix}{\mathrm{fix}}

\begin{document}
\title[Lattice of weakly continuous operators]{On the lattice structure of weakly continuous operators on the space of measures}
\author{Moritz Gerlach}
\address{Moritz Gerlach\\University of Ulm\\Institute of Applied Analysis\\89069 Ulm\\Germany}
\email{moritz.gerlach@uni-ulm.de}

\author{Markus Kunze}
\address{Markus Kunze\\DFG Research Training Group 1100\\University of Ulm\\89069 Ulm\\Germany}
\email{markus.kunze@uni-ulm.de}
\begin{abstract}
Consider the lattice of bounded linear operators on the space of Borel measures on a Polish space. We prove
that the operators which are continuous with respect to the weak topology induced by the bounded measurable functions form a sublattice
that is lattice isomorphic to the space of transition kernels.
As an application we present a purely analytic proof of Doob's theorem concerning stability of transition semigroups.
\end{abstract}


\maketitle

\section{Introduction}

Classical function spaces, such as 
$L^p(\Omega,\Sigma,\mu)$, the space of $p$-integrable functions on a measure space $(\Omega,\Sigma,\mu)$,
come with a natural ordering, which renders them into Banach lattices.
While it is clear how to compute the modulus of functions, the situation changes when one considers
operators on Banach lattices, ordered by the cone of positive operators.
Indeed, regular operators on a general Banach lattice 
need not have a modulus unless the space is order complete.
Moreover, even if the modulus exists, it is in general not clear how to compute it.

The situation improves when one restricts to certain classes of operators. For instance, if
a regular operator $T$ on $L^p(\Omega,\Sigma,\mu)$ is of the form
\[ (Tf)(x) = \int_\Omega k(x,y)f(y) \dx\mu(y) \]
for a measurable functions $k:\Omega\times\Omega \to \R$,
the modulus of $T$ is given by
\[ (\lvert T\rvert f)(x) = \int_\Omega \lvert k(x,y)\rvert f(y) \dx\mu(y).\]
Operators of this form are often called kernel operators and can be identified with the
band generated by the finite rank operators \cite[Prop IV.9.6]{schaefer1974}.
Moreover, kernel operators are characterized by an additional continuity condition due to Bukhvalov \cite{bukhvalov1978}
and appear naturally in the theory of partial differential equations, see for example \cite{ouhabaz2005}.

In the study of Markov processes one is interested in certain operators on the space of measures which describe the evolution
of distributions.
The probabilistic literature also uses the name kernel operators for these operators since they are 
associated with a transition kernel in the sense of Definition \ref{def:kernel} below.
To distinguish them from kernel operators in the sense above, we call them 
\emph{weakly continuous operators} as they are also characterized by a continuity condition, see Lemma \ref{l.weakcont}.

In Section \ref{sec:lattice} we show that the modulus of a transition kernel is again a transition kernel.
Consequently, the weakly continuous operators form a lattice.
In Section \ref{sec:sublattice} we consider the weakly continuous operators as a subspace of all regular operators on the space
of measures and prove that they are a countably order complete sublattice.
Hence, as for kernel operators, the computation of lattice operations of weakly continuous operators reduces
to the corresponding lattice operations for their transition kernels.
However, in contrast to kernel operators, weakly continuous operators are not an ideal, see Example \ref{ex:noideal}.


As an application of our results we give a purely analytic proof of a version of Doob's theorem
concerning the stability of one-parameter semigroups operating 
on the space of measures, which is due to Stettner \cite[Thm 1]{stettner1994},
see also \cite{seidler1997}. 
Our strategy is similar to the one in \cite{gerlach2012}, however due to the abstract results of Sections \ref{sec:lattice}
and \ref{sec:sublattice}
the proof simplifies as we can work directly within the lattice of weakly continuous operators.
This complements recent results about mean ergodicity of semigroups of weakly continuous operators, see \cite{gerlach2013}.

\section{The Lattice of Weakly Continuous Operators}
\label{sec:lattice}
Throughout, $\Omega$ denotes a Polish space and $\mathscr{B}(\Omega)$ its  Borel $\sigma$-algebra.
We denote by $\mathscr{M}(\Omega), B_b(\Omega)$ and $C_b(\Omega)$ the spaces of signed measures on 
$\mathscr{B}(\Omega)$, the space of bounded, Borel-measurable functions on $\Omega$ and the space 
of bounded continuous functions on $\Omega$, respectively.
We denote by $\applied{\; \cdot\;}{\; \cdot\;}$ the duality between $B_b(\Omega)$ and $\mathscr{M}(\Omega)$.

\begin{defn}
\label{def:kernel}
	A \emph{transition kernel} on $\Omega$ is a map $k: \Omega\times\mathscr{B}(\Omega)\to \R$ such that
	\begin{enumerate}[(a)]
	\item $A \mapsto k(x, A)$ is a signed measure for every $x\in \Omega$ and
	\item $x \mapsto k(x,A)$ is a measurable function for every $A\in \Sigma$.
	\end{enumerate}
	The total variation of the measure $k(x,\;\cdot\;)$ is denoted by $\lvert k \rvert(x,\;\cdot\;)$.
	The transition kernel $k$ is called \emph{bounded} if $\sup_{x\in \Omega} \lvert k\rvert (x, \Omega) <\infty$.
\end{defn}

We order the transition kernels on $\Omega$ pointwise, i.e.\ $k_1\leq k_2$ if and only if $k_1(x,A) \leq k_2(x,A)$ for 
all $x \in \Omega$ and $A \in \mathscr{B}(\Omega)$. With this ordering, the transition kernels form a lattice, as we show in Proposition \ref{prop:abskernel} below.
In its proof, we use the \emph{strict topology} $\beta_0$ on $C_b(\Omega)$ that is defined as follows. Let $\mathscr{F}_0$ be 
the space of functions $\varphi$ on $\Omega$ which vanish at infinity, i.e.\ given $\eps > 0$ there exists a compact 
set $K$ with $\lvert \varphi (x)\rvert \leq \eps$ for all $x \in \Omega\setminus K$. 
The strict topology $\beta_0$ on $C_b(\Omega)$ is the locally convex topology generated by the set of
seminorms $\{ p_\varphi : \varphi \in \mathscr{F}_0\}$ where $p_\varphi(f) \coloneqq \Vert \varphi f\Vert_\infty$.
This topology is consistent with the duality, i.e.\ $(C_b(\Omega),\beta_0)'=\mathscr{M}(\Omega)$,
see \cite[Thm 7.6.3]{jarchow1981}, and
it coincides with the compact open topology on norm bounded subsets of $C_b(\Omega)$ \cite[Thm 2.10.4]{jarchow1981}.
Important for the following proof is especially the fact that the lattice operations on $C_b(\Omega)$ are $\beta_0$-continuous,
which follows from \cite[V 7.1]{schaefer1971} since in the strict topology the origin has a neighborhood base of solid sets.

\begin{prop}
\label{prop:abskernel}
	If $k : \Omega\times \mathscr{B}(\Omega) \to \R$ is a transition kernel, then also $\lvert k\rvert : \Omega \times \mathscr{B}(\Omega) \to [0,\infty)$
	is a transition kernel. 
\end{prop}

\begin{proof}
  It is well-known that on a Polish space $\Omega$ the space $C_b(\Omega)$ of bounded and continuous functions is norming 
  for the measures, see, e.g., \cite[Example 2.4]{kunze2011}.
  We thus have
     \begin{align}
		\label{eqn:variation}
		\lvert k\rvert (x,\Omega) = \sup_{\substack{f\in C_b(\Omega)\\\lVert f\rVert\leq 1}} \lvert \applied{f}{k(x,\;\cdot\;)}\rvert
	 \end{align}
	which remains obviously true if we replace $\Omega$ with a closed subset $F$ of $\Omega$.
	Now we construct a countable set $D$ independent of $x$ such that \eqref{eqn:variation} holds even if we take
	the supremum only over the set $D$. It then follows
	that $x \mapsto |k|(x,F)$ is measurable for every closed set $F$ as a supremum of countable many measurable functions.
	A monotone class argument 
	then shows that $x \mapsto |k|(x,A)$ is measurable for all $A \in \mathscr{B}(\Omega)$.

	So fix a closed set $F \subset \Omega$. By \cite[Thm 6.3]{kunze2011} 
	there exists a countable set $M \subset C_b(F)$ such that
	for all measures $\mu\in \mathscr{M}(F)$, $\mu\neq 0$, there exists $f \in M$ with $\applied{\mu}{f} \neq 0$.
	
	We denote by $S \coloneqq \lh_\Q M$ the linear span of $M$ with rational coefficients. Obviously, 
	the $\beta_0$-closure of $S$ is the same as the $\beta_0$-closure of $\lh M$, which is dense in $C_b(F)$.
	Indeed, if $g\in C_b(F)$ does not belong to the $\beta_0$-closure of $\lh M$, then, by the Hahn-Banach theorem,
    we find
	$\mu \in \mathscr{M}(F)$ such that $\applied{\mu}{f}=0$ for all $f\in \lh M$ whereas $\applied{\mu}{g}\neq 0$.
	Since $\mu$ vanishes in particular on the separating set $M$, 
	it follows that $\mu=0$ --- a contradiction.

	Now we define $D \coloneqq \{ f\wedge \mathds{1} \vee (-\mathds{1}) : f \in S\}$. Since $S$ is $\beta_0$-dense in
	 $C_b(F)$ and the lattice operations are $\beta_0$-continuous, $D$ is $\beta_0$-dense in the closed unit ball 
	 of $C_b(F)$.

	Let $x\in \Omega$ and $f\in C_b(F)$, $\lVert f\rVert \leq 1$. Given $\eps>0$ we find $g\in D$ such that 
	$\lvert \applied{g-f}{k(x,\;\cdot\;)} \rvert \leq \eps$. Hence,
	\begin{align*}
	\lvert \applied{f}{k(x,\;\cdot\;)}\rvert &\leq \lvert \applied{g}{k(x,\;\cdot\;)}\rvert +\eps 
	\leq \sup_{g\in D} \lvert \applied{g}{k(x,\;\cdot\;)}\rvert + \eps\\
	&\leq \sup_{\substack{g\in C_b(F)\\\lVert g\rVert\leq 1}} \lvert \applied{g}{k(x,\;\cdot\;)}\rvert + \eps = \lvert k\rvert (x,A)+
	    \eps.
	\end{align*}
	From this it follows that
	\[ \lvert k\rvert(x,A) = \sup_{f\in D} \lvert \applied{f}{k(x,\;\cdot\;)}\rvert\]
	as desired. This finishes the proof.
\end{proof}

To each bounded transition kernel $k$, we can associate an operator $T \in \mathscr{L}(\mathscr{M}(\Omega))$ by setting
\begin{equation}\label{eq.kernel}
(T\mu)(A)  \coloneqq \int_\Omega k(x,A)\, \dx\mu (x)\, .
\end{equation}

The following characterization of operators of this form follows from Propositions 3.1 and 3.5 of \cite{kunze2011}.

\begin{lem}\label{l.weakcont}
Let $T \in \mathscr{L}(\mathscr{M}(\Omega))$. The following are equivalent:
\begin{enumerate}[(i)]
\item There exists a bounded transition kernel $k$ such that $T$ is given by \eqref{eq.kernel}.
\item The norm adjoint $T^*$ of $T$ leaves $B_b(\Omega)$ invariant.
\item The operator $T$ is continuous in the $\sigma(\mathscr{M}(\Omega), B_b(\Omega))$ topology.
\end{enumerate}
\end{lem}

\begin{defn}
We call an operator \emph{weakly continuous} if it satisfies the equivalent conditions in Lemma \ref{l.weakcont}.
In this case, the transition kernel $k$ from (i) of the Lemma is called the \emph{associated transition kernel}.
We write $\mathscr{L}(\mathscr{M}(\Omega), \sigma)$ for the space of weakly continuous operators.
\end{defn}

Now let $T_1, T_2$ be weakly continuous operators  with associated transition kernels $k_1, k_2$. Noting that 
$\applied{\mathds{1}_A}{T_j\delta_x} = k_j(x,A)$ for $j=1,2$, we see that $T_1 \leq T_2$ as operators 
on $\mathscr{M}(\Omega)$ if and only if $k_1 \leq k_2$ as transition kernels. Thus the correspondence between a weakly continuous
operator and its transition kernel is actually a lattice isomorphism. 
We thus obtain immediately from Proposition \ref{prop:abskernel} the following result.

\begin{thm}
\label{thm:weakcontlattice}
	The space $\mathscr{L}(\mathscr{M}(\Omega),\sigma)$ is a lattice in its natural ordering inherited from
	$\mathscr{L}(\mathscr{M}(\Omega))$.
\end{thm}

The question arises whether also the $\sigma (\mathscr{M}(\Omega), C_b(\Omega))$-continuous operators form a lattice.
Equivalently, if the norm adjoint of a weakly continuous operator $T$ with transition kernel $k$ leaves 
the space $C_b(\Omega)$ invariant, does the same hold for the operator given by the kernel $\lvert k\rvert$? 
The following example shows that this is not the case.

\begin{example}
We consider the set $\Omega = (-\N)\cup \N \cup \{\infty\}$, where the neighborhoods of the 
extra point $\infty$ are exactly the sets which contain a set of the form 
$ \{\infty\} \cup \{ n, n+1, \dots\} \cup \{-n, -(n+1),\ldots\}$
for some $n\in\N$, whereas all other points are isolated. Note that $\Omega$ is homeomorphic with the space $\{0, \pm n^{-1}: n \in \N\}$ endowed with 
the topology inherited from $\R$. Thus $\Omega$ is Polish. We also note that a bounded function 
$f: \Omega \to \R$ is continuous if and only if $f(n) \to f(\infty)$ and also $f(-n) \to f(\infty)$
as $n\to \infty$.  

Now we define the transition kernel
\[ k(n,\;\cdot\;) \coloneqq \begin{cases} \delta_n  - \delta_{n+1}  &n\in \N\\
									0 & n\in (-\N)\cup \{ \infty\}. \end{cases} \]
Then
\[ \lvert k\rvert(n,\;\cdot\;) = \begin{cases} \delta_n + \delta_{n+1} &n\in \N\\
									0 & n \in (-\N)\cup \{\infty\}. \end{cases} \]									
Let $T, U\in \mathscr{L}(\mathscr{M}(\Omega),\sigma)$ denote the operators associated with $k$ and $\lvert k\rvert$, 
respectively. Then $T^*C_b(\Omega) \subset C_b(\Omega)$. However, $U^*$ maps the continuous function $\one_\Omega$
to the function $2\one_\N$ which is not continuous. 
This shows that the $\sigma(\mathscr{M}(\Omega), C_b(\Omega))$-continuous operators do not form a sublattice 
of the $\sigma(\mathscr{M}(\Omega),B_b(\Omega))$-continuous operators. Moreover, there exists 
no modulus of $T$ in the $\sigma(\mathscr{M}(\Omega),C_b(\Omega))$-continuous operators. Indeed, if $S$ was 
such a modulus, then the transition kernel of $S$ has to
coincide with $\lvert k\rvert$ on $\N \cup(-\N)$. In particular $S^*\one_\Omega (n) = 2 \one_{\N}(n)$ for $n\neq \infty$.
But this shows that $S^*\one_\Omega$ cannot be continuous --- a contradiction.
\end{example}

\section{Weakly Continuous Operators as a Sublattice of $\mathscr{L}(\mathscr{M}(\Omega))$}
\label{sec:sublattice}

Since $\mathscr{M}(\Omega)$ is a $L$-space, every bounded linear operator on $\mathscr{M}(\Omega)$ is regular and 
$\mathscr{L}(\mathscr{M}(\Omega))$ forms a Banach lattice with respect to the natural ordering, see \cite[Thm IV 1.5]{schaefer1974}.
Thus, every weakly continuous operator has a modulus in $\mathscr{L}(\mathscr{M}(\Omega))$. A natural question is whether this modulus can be different from 
the modulus in the space of weakly continuous operators.
This is not the case. In the following we show that the weakly continuous operators form a sublattice of
$\mathscr{L}(\mathscr{M}(\Omega))$.  More precisely,
we show that for $T\in \mathscr{L}(\mathscr{M}(\Omega),\sigma)$, the positive part $T^+$, taken in the vector lattice $\mathscr{L}(\mathscr{M}(\Omega))$,
is again weakly continuous. Moreover, if $k$ is the transition kernel associated to $T$,
then $T^+$ is associated to the transition kernel $k_+ = (\lvert k\rvert-k)/2$.

We recall that the positive part within $\mathscr{L}(\mathscr{M}(\Omega))$ of a weakly continuous operator $T$ with associated transition kernel $k$ is given by
\begin{align}
\label{eqn:pospart}
 T^+\mu = \sup_{0\leq \nu \leq \mu} T\nu = \sup_{0\leq \nu \leq \mu} \int_\Omega k(x,\;\cdot\;) \dx\nu(x) = \sup_{0\leq g\leq \mathds{1}} \int_\Omega g(x) k(x, \;\cdot\;)\dx\mu(x) 
 \end{align}
for every positive measure $\mu$.

\begin{lem}
\label{lem:alphaBn}
	Let $k$ be a transition kernel, $\alpha > 0$ and let $\mathscr{U}=\{ B_n : n\in\N\}$ be a countable basis of the topology on $\Omega$ that is closed under finite unions.
	Then 
	\[\{ k_+( \;\cdot\;,\Omega) > \alpha \} = \bigcup_{n\in\N} \{ k(\;\cdot\;,B_n) > \alpha\}.\]
\end{lem}
\begin{proof}
If $k(x,B_n)>\alpha$ for some $x\in \Omega$ and $n\in\N$, then clearly $k_+(x,\Omega) \geq k_+(x,B_n)\geq k(x,B_n)>\alpha$. This shows the inclusion ``$\supset$''.

Conversely, let $x\in \Omega$ with $k_+(x,\Omega)>\alpha$ be given. We consider the Hahn decomposition 
$\Omega= \Omega_+ \cup \Omega_-$
of the measure $k(x,\;\cdot\;)$. By assumption $k(x,\Omega_+)=k_+(x,\Omega)>\alpha$.
Since the measure $k(x,\;\cdot\;)$ is regular, there exists an open superset $U \supset \Omega_+$ with $k(x,U)>\alpha$.
Since $\mathscr{U}$ is closed under finite unions, using the regularity of $k(x,\;\cdot\;)$ again, we find
a base set $B_n \in \mathscr{U}$ with $p(x,B_n)>\alpha$.
\end{proof}

\begin{lem}
\label{lem:keylemma}
	Let $T\in\mathscr{L}(\mathscr{M}(\Omega),\sigma)$ with associated transition kernel $k$.
	Let $\alpha>0$ and $A\in \mathscr{B}(\Omega)$ such that $\alpha \mathds{1}_A < k_+(\;\cdot\;,\Omega)$. Then
	\[ (T^+\mu_A)(\Omega) \geq \alpha \mu(A) \]
	where $\mu_A$ denotes the measure $\mu(A\cap \;\cdot\;)$.
\end{lem}
\begin{proof}
	Let $(B_n)_{n\in\N}$ be a countable basis of the topology on $\Omega$ that is closed under finite unions.
	We define $E_n \coloneqq A \cap \{ k(\;\cdot\;, B_n)>\alpha\}$ for $n\in\N$. Then Lemma~\ref{lem:alphaBn} yields that
	\[ A = A\cap \{ k_+( \;\cdot\;,\Omega)>\alpha\} = \bigcup_{n\in\N} E_n.\]
	Defining $\Omega_1 \coloneqq E_1$ and $\Omega_n \coloneqq E_n \setminus (\cup_{k<n} E_k)$ for $n>1$ we obtain a 
	decomposition of $A$ in disjoint sets.
	Fix $\eps>0$.
	By the regularity of $\mu$ we find an index $N\in\N$ with
	\[ \mu\biggl(\bigcup_{n\leq N} \Omega_n \biggr) \geq \mu(A)-\frac{\eps}{\alpha}.\]
	We now refine the sets $B_1, \ldots, B_N$ further. We find disjoint Borel sets $\tilde B_1, \ldots, \tilde B_M$ such that
	(i) given $m\leq M$ and $n\leq N$ the set $\tilde B_m$ is either contained in $B_n$ or disjoint from $B_n$
	and (ii) we have
	\[\bigcup_{m\leq M} \tilde B_m = \bigcup_{n\leq N} B_n .\]
	We let $N(m) \coloneqq \{ n\leq N  : \tilde B_m \subset B_n\}$ 
	so that $B_n$ is the disjoint union of those $\tilde B_m$ where $n\in N(m)$. 
	By choosing $g$ in \eqref{eqn:pospart} as the characteristic of the set $\cup_{n\in N(m)} \Omega_n$, we find that
	\begin{align*}
	(T^+\mu_A)(\Omega) &\geq \sum_{m=1}^M (T^+\mu_A)(\tilde B_m)
	\geq \sum_{m=1}^M \int_{\bigcup_{n\in N(m)} \Omega_n} k(x,\tilde B_m) \dx\mu(x).
	\end{align*}
	Since the sets $\Omega_n$ as well as the sets $\tilde B_m$ are disjoint, we have that
	\begin{align*}
	\sum_{m=1}^M \int_{\bigcup_{n\in N(m)} \Omega_n} k(x,\tilde B_m) \dx\mu(x)
	&=\sum_{m=1}^M \sum_{n\in N(m)} \int_{\Omega_n} k(x,\tilde B_m) \dx\mu(x)\\
	&= \sum_{n=1}^N \sum_{\substack{m\leq M\\n\in N(m)}} \int_{\Omega_n} k(x,\tilde B_m) \dx\mu(x)\\
	&= \sum_{n=1}^N \int_{\Omega_n} k(x,B_n) \dx\mu(x).
	\end{align*}
	As $k(\;\cdot\;,B_n) > \alpha$ on $\Omega_n$, we conclude that
	\begin{align*}
	\sum_{n=1}^N \int_{\Omega_n} k(x,B_n) \dx\mu(x) > \alpha \mu\biggl(\bigcup_{n\leq N} \Omega_n \biggr) \geq \alpha \mu(A) - \eps
	\end{align*}
	which completes the proof.
\end{proof}

\begin{thm}
\label{thm:main}
	The $\mathscr{L}(\mathscr{M}(\Omega),\sigma)$ is a sublattice of $\mathscr{L}(\mathscr{M}(\Omega))$.
\end{thm}
\begin{proof}
	Let $T\in\mathscr{L}(\mathscr{M}(\Omega),\sigma)$ with associated transition kernel $k$.
	We denote by $S$ the weakly continuous operator with transition kernel $k_+$ and we prove that $T^+ = S$.
	To that end, let $\mu>0$.
	Since it follows easily from \eqref{eqn:pospart} that $T^+\mu \leq S\mu$, it suffices to show that $(T^+\mu)(\Omega) = (S\mu)(\Omega)$.
	Let $\eps>0$ and 
	\[ f = \sum_{j=1}^M \alpha_j \mathds{1}_{A_j}\]
	be a simple function with coefficients $\alpha_j >0$ and pairwise disjoint sets $A_j \in \mathscr{B}(\Omega)$ such that
	$f(x) < k_+(x,\Omega)$ for all $x\in \Omega$ and
	\[ \int_\Omega \biggl(k_+(x,\Omega) - f(x)\biggr) \dx\mu(x) < \eps.\]
	Lemma \ref{lem:keylemma} yields that
	\begin{align*}
		(T^+\mu)(\Omega) &\geq \sum_{j=1}^M (T^+ \mu_{A_j}) (\Omega) \geq \sum_{j=1}^M \alpha_j \mu(A_j)\\
		&= \int_\Omega f(x) \dx\mu(x) \geq \int_\Omega k_+(x,\Omega)\dx\mu(x) - \eps\\
		&= (S\mu) (\Omega)-\eps.
	\end{align*}
	Hence $T^+\mu = S\mu$ and thus, since $\mu$ was arbitrary, $T^+ = S \in \mathscr{L}(\mathscr{M}(\Omega),\sigma)$.
\end{proof}

In contrast to the situation for kernel operators, which can be identified with the band generated by the finite rank operators
as described in the introduction, the weakly continuous operators are not a band in $\mathscr{L}(\mathscr{M}(\Omega))$. 
The following example shows that they are not even an ideal.

\begin{example}
\label{ex:noideal}
	Let $\Omega$ be a Polish space that admits atomless measures, e.g.\ $\Omega = \R$. Let
	$P:\mathscr{M}(\Omega) \to \mathscr{M}(\Omega)$ denote the band projection onto the band of atomless
	measures and define $\phi \coloneqq P^*\mathds{1}$.
	Then $0<\phi \leq \mathds{1}$ and $\phi \in \mathscr{M}(\Omega)^*\setminus B_b(\Omega)$ since $\applied{\phi}{\delta_x}=0$ for all $x\in \Omega$.
	For a measure $\mu > 0$ consider the positive rank one operator
	$T \coloneqq \phi \otimes \mu$ on $\mathscr{M}(\Omega)$. Then $T \leq \mathds{1}\otimes \mu \in \mathscr{L}(\mathscr{M}(\Omega),\sigma)$
	but $T^*\mathds{1} = \mu(\Omega) \phi \not\in B_b(\Omega)$ and hence $T\not\in\mathscr{L}(\mathscr{M}(\Omega),\sigma)$.
\end{example}

We conclude this section with an investigation of order completeness of 
the sublattice $\mathscr{L}(\mathscr{M}(\Omega),\sigma)$. 
We prove that this space is $\sigma$-order complete but not order complete.
Let us start with a well-known lemma.

\begin{lem}
\label{lem:monlimit}
	Let $(\mu_n)$ be an increasing sequence of positive measures on $\mathscr{B}(\Omega)$
	and $\nu\in\mathscr{M}(\Omega)$ such that
	$\mu_n \leq \nu$ for all $n\in\N$.
	Then $\mu \coloneqq \sup \mu_n$ is given by 
	$\mu(A) = \sup \mu_n(A)$ for all $A\in\mathscr{B}(\Omega)$.
\end{lem}
\begin{proof}
	Let $f_n$ denote the density of $\mu_n$ with respect to $\nu$ and define $f\coloneqq \sup f_n$. Then
	\[ \int_A f \dx\nu = \sup_{n\in\N} \int_A f_n \dx\nu = \sup_{n\in\N} \mu_n(A)\]
	for all $A\in\mathscr{B}(\Omega)$ by the monotone convergence theorem. Therefore, the mapping $A\mapsto \sup \mu_n(A)$
	defines a measure on $\mathscr{B}(\Omega)$ and thus $(\sup \mu_n)(A) = \sup \mu_n(A)$ for all $A\in\mathscr{B}(\Omega)$.
\end{proof}

\begin{thm}
	Let $(T_n) \subset \mathscr{L}(\mathscr{M}(\Omega),\sigma)$ be a sequence of weakly continuous operators 
	that is order bounded by an element of $\mathscr{L}(\mathscr{M}(\Omega))$.
	Then $\sup T_n$ exists in $\mathscr{L}(\mathscr{M}(\Omega))$ and is weakly continuous.
\end{thm}
\begin{proof}
	Since $\mathscr{M}(\Omega)$ is order complete, $S \coloneqq \sup T_n$ exists in $\mathscr{L}(\mathscr{M}(\Omega))$.
	It remains to show that $S\in\mathscr{L}(\mathscr{M}(\Omega),\sigma)$.

	Note that by Theorem \ref{thm:main}, $T_1\vee\dots\vee T_n$ is again weakly continuous. Thus, 
	replacing $T_n$ by $T_1\vee \dots \vee T_n - T_1$ and $S$ by $S-T_1$,
	we may assume that $(T_n)$ is increasing and $T_n \geq 0$ for all $n\in\N$.
	We denote by $k_n$ the transition kernel associated with $T_n$. For $x\in \Omega$ and $A\in \mathscr{B}(\Omega)$ we define
	\[ k(x,A) \coloneqq \sup_{n\in\N} k_n(x,A) \leq (S\delta_x)(A) \leq \lVert S\rVert.\]
	Then $k(\;\cdot\;,A)$ is measurable for all $A\in\mathscr{B}(\Omega)$
	and $k(x,\;\cdot\;)$ is a measure by Lemma \ref{lem:monlimit}.
	Hence, $k$ is a bounded transition kernel.
	Since $(T_n)$ is increasing, we have for all $\mu\in\mathscr{M}(\Omega)_+$ that
	\[ (\sup T_n) \mu =  \sup (T_n \mu) = \sup_{n\in\N} \int_\Omega k_n(x, \; \cdot\; )\dx\mu(x) 
	= \int_\Omega k(x,\;\cdot\;)\dx\mu(x),\]
	where the last identity follows from Lemma \ref{lem:monlimit} and the monotone convergence theorem.
\end{proof}

The following example shows that $\mathscr{L}(\mathscr{M}(\Omega),\sigma)$ is not order complete, i.e.\ 
not every order bounded set has a supremum.

\begin{example}
\label{ex:notordercomplete}
	Let $\Omega$ be a Polish space such that there exists an unmeasurable set $E\subset \Omega$, e.g.\ $\Omega=\R$.
	Let $\mu\in\mathscr{M}(\Omega)$ a probability measure.
	For each $x\in E$ we consider the weakly continuous rank one operator $T_x \coloneqq \mathds{1}_{\{x\}} \otimes \mu$. 
	Then the set $\mathscr{T} \coloneqq \{ T_x : x\in E\}$ is dominated by $\mathds{1}\otimes \mu$.
	We show that $S\coloneqq \sup_{x\in E} T_x$ is not weakly continuous. 
	If $x \in \Omega\setminus E$ and $P$ is the band projection onto $\{\delta_x\}^{\bot\bot}$, then $S(I-P)$ is 
	an upper bound of $\mathscr{T}$ and hence $(S\delta_x)(\Omega) \leq (S(I-P)\delta_x)(\Omega)=0$.
	If $x\in E$, then $(S\delta_x)(\Omega)\geq (T_x\delta_x)(\Omega) = 1$ for all $x\in E$.
	Therefore, $S^*\mathds{1}$ is not a measurable function. This proves the claim.
\end{example}

\section{Stability of Transition Semigroups}
\label{sec:doob}

As a consequence of the lattice structure of $\mathscr{L}(\mathscr{M}(\Omega),\sigma)$, we obtain
a version of Doob's theorem on convergence of semigroups on the space of measures to a projection
onto their fixed space.

We start with recalling some terminology from \cite{daprato1996}.

\begin{defn}
	Let $\mathscr{T} =(T(t))_{t\geq 0} \subset \mathscr{L}(\mathscr{M}(\Omega),\sigma)$ be a semigroup of weakly continuous operators
	with transition kernels $k_t$ such that $T(0)=I$.
	Then $\mathscr{T}$ is called \emph{Markovian} if $k_t(x,\;\cdot\;)$ is a probability measure for all
	$x\in \Omega$ and $t\geq 0$.
	If
	\[ \lim_{t\to 0} (T(t)f)(x) = f(x)\]
	for all $f\in C_b(\Omega)$ and all $x\in\Omega$, then $\mathscr{T}$ is said to be \emph{stochastically continuous}.
	For $t_0>0$, the semigroup $\mathscr{T}$ is called $t_0$-regular if 
	the measures $k_{t_0}(x,\;\cdot\;)$ and $k_{t_0}(y,\;\cdot\;)$ are equivalent for all $x,y\in\Omega$.
\end{defn}

\begin{rem}
\label{rem:t0-regular}
	If $\mathscr{T}=(T(t))_{t\geq 0}$ is $t_0$-regular for some $t_0 >0$, then $\mathscr{T}$ is $s$-regular for all $s\geq t_0$
	and the measures $k_s(x,\;\cdot\;)$ and $k_t(y,\;\cdot\;)$ are equivalent for all $s,t\geq t_0$ and $x,y\in \Omega$.
	Indeed, for $A\in\mathscr{B}(\Omega)$ and $r>0$ we have
	\begin{align}
	 k_{t_0+r}(x,A) = (T(t_0 +r)\delta_x)(A) = \int_\Omega k_{t_0}(y,A) k_{r}(x,\dx y).
	 \label{eqn:semilaw}
	 \end{align}
	Thus,
	$k_{t_0+r}(x,\;\cdot\;) \ll k_{t_0}(y,\;\cdot\;)$ for all $r>0$, $x\in\Omega$ and $y\in\Omega$.
	Conversely, if 
	\[ (T(t_0+r)\delta_x)(A) =k_{t_0+r}(x,A)  = 0\]
	for some $A\in\mathscr{B}(\Omega)$, it follows from \eqref{eqn:semilaw} that
	$k_{t_0}(y,A)=0$ for some and hence all $y\in\Omega$.
\end{rem}

Our main tool for the proof of Doob's theorem is a generalized version of a theorem by Greiner \cite[Kor.\ 3.9]{greiner1982}
that can be formulated as follows, see \cite[Prop 4.1]{gerlach2012b}.
Recall that a semigroup on a Banach lattice $E$ is said to be \emph{irreducible} if 
the only closed ideals in $E$ that are invariant under the action of every operator of the semigroup 
are $E$ and $\{0\}$.

\begin{thm}
\label{thm:greiner}
	Let $\mathscr{T} = (T(t))_{t \geq 0}$ be a positive, bounded and irreducible $C_0$-semigroup on 
	a Banach lattice with order continuous norm $E$ such that
	$\fix(T(t))$ is independent of $t>0$ and nontrivial.
	Assume that $T(r)\wedge T(s) >0$ for some $r>s\geq 0$.
	Then there exists a strictly positive $x' \in \fix(\mathscr{T'})$ and a quasi-interior point $e \in \fix(\mathscr{T})$ of $E_+$
	such that
	\[ \lim_{t\to\infty} T(t)x = \applied{x'}{x} e\]
	for all $x\in E$.
\end{thm}

A consequence of this theorem for semigroups of kernel operators can be found in \cite{arendt2008}.

Now we are able to prove the announced stability result that is originally due to Stettner \cite[Thm 1]{stettner1994}
who gave a probabilistic proof.

\begin{thm}[Doob]
\label{thm:doob}
	Let $\mathscr{T}=(T(t))_{t\geq 0} \subset \mathscr{L}(\mathscr{M}(\Omega),\sigma)$ be a stochastically continuous
	Markovian semigroup and let $\mu \in\mathscr{M}(\Omega)$ be an invariant probability measure. If $\mathscr{T}$ is
	$t_0$-regular for some $t_0>0$, then
	\[ \lim_{t\to\infty} T(t)\nu = \nu(\Omega) \cdot \mu\]
	in the norm of $\mathscr{M}(\Omega)$ as $t\to\infty$.
	Moreover, $\mu$ is the unique invariant probability measure 
	and equivalent to all $k_t(x,\;\cdot\;)$ for $t\geq t_0$ and $x\in\Omega$.
\end{thm}
\begin{proof}
	Let us denote by $E\coloneqq \{\mu\}^{\bot\bot}$ the band generated by the invariant measure $\mu$, which 
	consists precisely of those measures that are absolutely continuous with respect to $\mu$.
	In view of the equality
	\[ \mu = T(t)\mu = \int_\Omega k_{t}(x,\;\cdot\; )\dx\mu(x) \]
	it follows from the $t_0$-regularity of $\mathscr{T}$ 
	and Remark \ref{rem:t0-regular} 
	that $\mu$ is equivalent to $k_t(x,\;\cdot\;)$ for all $x\in\Omega$ and $t\geq t_0$.
	Hence, for every measure $\nu>0$,
	\[ T(t_0)\nu= \int_\Omega k_{t_0}(x,\;\cdot\;) \dx\nu(x)\]
	is equivalent to $\mu$, so that  $T(t_0)\nu \in E$ for every measure $\nu >0$.
	Replacing $\nu$ by $T(t_0)\nu$, it therefore suffices to show that
	\[ \lim_{t\to\infty} T(t)\nu = \nu(\Omega) \cdot \mu\]
	for all $\nu\in E$.

	Let $S(t) \coloneqq T(t)_{\mid E}$ denote the restriction of $T(t)$ to $E$. Then
	the semigroup $\mathscr{S} \coloneqq (S(t))_{t\geq 0}$ is clearly contractive and positive.
	Moreover, it is strongly continuous by \cite[Thm 4.6]{hille2009}.
	Fix $r>s\geq t_0$.  Since the measures $k_s(x,\;\cdot\;)$ and $k_r(x,\;\cdot\;)$ are equivalent for all $x\in \Omega$, 
	they cannot by disjoint. Hence, the operator 
	\[ S\coloneqq S(s) \wedge S(r) = ( T(s)\wedge T(r) )_{\mid E}\]
	is not zero
	as, by Theorem \ref{thm:weakcontlattice}, it is weakly continuous and given by the transition kernel 
	$q(x, \;\cdot\;) \coloneqq k_s(x,\;\cdot\;) \wedge k_r(x,\;\cdot\;)$ satisfying $q(x,\Omega)>0$ for all $x\in \Omega$.

	In order to prove that the fixed space $\fix(S(t))$ is independent of $t>0$, we first observe that
	for $t\geq t_0$ the operator $S(t)$ is expanding (or strongly positive), i.e.\ $S(t)\nu$ is a
	quasi-interior point of $E_+$ for every $\nu\in E$, $\nu >0$. Indeed, for $A\in \mathscr{B}(\Omega)$ with $\mu(A)>0$
	we know that $k_t(x,A)>0$ for all $x\in \Omega$ and therefore
	\[ (S(t)\nu)(A) = \int_\Omega k_t(x,A) \dx\nu(x) >0\]
	for every measure $\nu >0$.
	Since the semigroup $\mathscr{S}$ consists of expanding operators, we have that
	\[ \sigma_p(G) \cap i\R \subset \{0\} \]
	where $G$ denotes the generator of $\mathscr{S}$.
	This can be seen as in the proof of \cite[Thm 3.1]{gerlach2012b}, where it was assumed that $\mathscr{S}$ consists
	of kernel operators. But the same proof works for semigroups of expanding operators, see \cite[Rem 3.5(b)]{gerlach2012b}.
	Finally, by \cite[IV 3.8]{nagel2000} one has that
	\[ \fix(S(t)) = \overline{ \lh}_{n\in\Z} \ker\left(\frac{2\pi i n}{t} - G\right) = \ker G\]
	for all $t>0$.
	Since $\mathscr{S}$ is irreducible as it consists of expanding operators, we proved that
	$\mathscr{S}$ satisfies the assumptions of Theorem \ref{thm:greiner}
	and the assertion follows.
\end{proof}

	Using a discrete version of Theorem \ref{thm:greiner}, see  \cite[Prop 4.1]{gerlach2012b},
	one obtains the following result for Markov chains.
\begin{thm}
	Let $T\in \mathscr{L}(\mathscr{M}(\Omega),\sigma)$ with transition kernel $k$
	such that the measures $k(x,\;\cdot\;)$ and $k(y,\;\cdot\;)$ are equivalent probability measures for $x,y\in \Omega$.
	If there exists an invariant probability measure $\mu$, then $\lim_{n\to\infty} T^n \nu = \nu(\Omega)\cdot \mu$
	for every measure $\nu$ in the norm of $\mathscr{M}(\Omega)$.
\end{thm}

\end{document}